\documentclass[11pt,reqno]{amsart}
\usepackage{amsaddr}
\usepackage[pdftex]{graphicx}
\usepackage{pgfplots}
\usepackage{orcidlink}
\usepackage{geometry}
\usepackage{stix}
\usepackage{subcaption}
\usepackage{caption}
\usepackage{amsmath,amsthm}
\usepackage{cite}
\usepackage{hyperref}
\hypersetup{
    colorlinks=true, 
    linktoc=all,     
    allcolors=blue,  
}
\usepackage{url}
\newtheorem{theorem}{Theorem}[section]
\newtheorem{lemma}{Lemma}[section]
\newtheorem{proposition}{Proposition}[section]
\newtheorem{corollary}{Corollary}[section]
\newtheorem{conjecture}{Conjecture}[section]
\theoremstyle{definition}
\newtheorem{definition}{Definition}[section]
\newtheorem{example}{Example}[section]
\newtheorem{remark}{Remark}[section]
\numberwithin{equation}{section}
\pgfplotsset{compat=1.18}
\definecolor{maingray}{rgb}{60,60,60}
\definecolor{darkgray}{rgb}{96,96,96}
\definecolor{auxgray}{rgb}{160,160,160}
\definecolor{softgray}{rgb}{120,120,120}
\definecolor{wewdxt}{rgb}{0.45,0.45,0.45}
\definecolor{zzttqq}{rgb}{0.6,0.2,0.}
\definecolor{qqwuqq}{rgb}{0.,0.5,0.}

\begin{document}
\pagenumbering{arabic}

\title[Poncelet Triangles and  Invariants]{A Symmetric Polynomial Approach to Poncelet Triangles Inscribed in a circle and circumscribed about Central Conics}
\author{Mohammad Hassan Murad\orcidlink{0000-0002-8293-5242}}
\address{Department of Mathematics\\
The University of Texas at Arlington, Arlington, TX, USA}
\email{mohammad.murad2@uta.edu}

\begin{abstract}
We develop a symmetric-polynomial framework for the study of one-parameter families of Poncelet triangles. For triangles inscribed in the unit circle and circumscribed about a central conic, we show that the elementary symmetric polynomials of the vertices depend linearly on a single complex parameter. It follows that every symmetric rational function of the vertices is a rational function of this parameter, providing a unified algebraic approach to the investigation of geometric invariants. As applications, we recover several known invariance results and establish new ones involving angles, orthic triangles, tangential triangles, polar circles, and other classical constructions. We also obtain explicit formulas and loci for several associated geometric objects, leading to new invariance phenomena within Poncelet families.
\end{abstract}


\maketitle{}

\section{Introduction}
\noindent
One of the central themes in the geometry of Poncelet polygons is the existence of geometric \,invariant quantities. Numerous invariants have been discovered recently for families of Poncelet triangles, concerning classical centers, areas, associated circles and triangles.

The present paper is the third in a series devoted to the geometry of Poncelet triangles and \,central conics. In the first paper \cite{Dragovic-Murad2026}, we established generalized Chapple--Euler relations together with \,several related geometric properties. In the second paper \cite{Murad2026b}, using Marden's theorem, we \,developed an explicit parametrization for the associated Poncelet triangles inscribed in the unit \,circle and circumscribed about a central conic in elementary symmetric polynomials of its vertices. See Theorem \ref{thm:sympara}. This parametrization provides the algebraic foundation for the present paper and allows the study of geometric invariants to be reduced to the analysis of a single complex parameter.

This viewpoint has played an important role in the study of invariant quantities associated with families of triangles circumscribed about ellipses contained in their common circumcircle. See, for example, 
\cite{Helmanetal.2022,Helmanetal.2023,Garciaetal.2026}.

The purpose of the present paper is to demonstrate that this parametrization provides a unified and effective framework for the study of geometric invariants associated with Poncelet families. The central observation is simple but powerful. Since the elementary symmetric polynomials $\sigma_1,\sigma_2,\sigma_3$ generate the algebra of symmetric polynomials, every symmetric polynomial in the vertices of a Poncelet triangle can be expressed as a polynomial in $\sigma_1,\sigma_2,\sigma_3$. Upon substituting the symmetric parametrization, these become polynomials in the single parameter $\lambda$, while symmetric rational functions become rational functions of $\lambda$. Consequently, every problem concerning a symmetric geometric quantity associated with a Poncelet family can be reduced to the study of a single complex parameter.

A subtle point deserves mention. Every triangle in a Poncelet family determines a unimodular parameter $\lambda$, but the converse is not true in general \cite{Murad2026b}. Accordingly, we define
\[
\Lambda=
\{\lambda\in\mathbb T:
\text{$\lambda$ corresponds to a triangle in the Poncelet family }\mathcal P\}.
\]
Thus $\Lambda\subseteq\mathbb T$. When both foci satisfy $a_1,a_2\in\mathbb D$, it follows from the results of Daepp et al. \cite{daepp2002ellipses} that every unimodular parameter generates a Poncelet triangle, so that $\Lambda=\mathbb T$.

Throughout the paper whenever we say that a quantity is \emph{invariant} throughout the family $\mathcal{P}$, we mean that it is independent of the parameter $\lambda\in\Lambda$.

\begin{theorem}[Reduction Principle]
Every symmetric rational function of the vertices of a triangle inscribed in a circle and circumscribed about a central conic is a rational function of the parameter $\lambda$ and the foci of the inconic. Consequently, a geometric quantity remains invariant throughout the family if and only if the corresponding rational function is constant on $\Lambda$.
\end{theorem}

This reduction transforms many geometric problems into elementary algebraic ones. Rather than proving each invariant separately, one first expresses the relevant geometric quantity as a symmetric polynomial or rational function of the vertices and then substitutes the symmetric parametrization. The problem of deciding whether the quantity remains invariant is thereby reduced to determining whether the resulting rational function is constant on $\Lambda$. In this sense, the symmetric parametrization serves as a general computational tool for the invariants of Poncelet triangles.

The effectiveness of this approach is illustrated through a variety of applications. We recover several previously known invariants by short and uniform proofs, including the invariance of the sum of the squares of the side lengths and the total area of the power circles established in \cite{Dragovic-Murad2026}. More importantly, we obtain several new results for orthic triangles, midpoint configurations, angle identities, tangential triangles, polar triangles, power circles, and de Longchamps circles. These applications demonstrate that the symmetric parametrization is not merely a description of Poncelet triangles but a practical framework for discovering and proving geometric invariants.

A special case of one of the recovered invariants, originally observed by Dan Reznik, was recently established in \cite{CelikDuguinGuoLuoSpinelliZeytuncuZhu2026} using degree-$3$ Blaschke products under the assumption that both foci of the associated ellipse lie inside the circumcircle. Our approach provides a considerably shorter proof of a stronger form of the observation and applies uniformly to all central conics arising from $3$-Poncelet pairs, including ellipses with both foci lying outside the circumcircle and hyperbolas with one focus lying  outside the circumcircle. See the proof in Subsection \ref{sec:powcirc}.

The paper is organized as follows. Section \ref{sec:sympolfra} recalls the symmetric parametrization and the \,necessary preliminaries from \cite{Murad2026b}. Section \ref{sec:geomapp} and the remaining sections apply this framework to \,derive and study a broad collection of related geometric invariants associated with Poncelet \,triangles, including orthic and tangential triangles; polar, power and de Longchamps circles.

\section{The Symmetric Polynomial Framework}\label{sec:sympolfra}
\noindent
Throughout the paper, we use
\[
\mathbb D=\{z\in\mathbb C:|z|<1\},\qquad
\mathbb T=\{z\in\mathbb C:|z|=1\}=\partial\mathbb D,
\]
and $z_1,z_2,z_3\in\mathbb T$ denote the vertices of a Poncelet triangle. We write
\[
\sigma_1=z_1+z_2+z_3,\qquad
\sigma_2=z_1z_2+z_2z_3+z_3z_1,\qquad
\sigma_3=z_1z_2z_3
\]
for the elementary symmetric polynomials of the vertices.

$\mathcal P$ will denote the family of triangles inscribed in a circle and circumscribed about a central conic. In sections where the circumcircle is normalized to the unit circle, we write $\mathbb T$ for the circumcircle.

The following theorem was proved using Marden's theorem in \cite{Murad2026b}.

\begin{theorem}[Parametrization of Poncelet Triangles]
\label{thm:sympara}
Let $z_1,z_2,z_3\in\mathbb T$ be the vertices of a triangle circumscribed about a central conic with foci $a_1, a_2 \in\mathbb C$. Then there exists $\lambda\in\mathbb T$ such that
\begin{align}
\sigma_1
&=
a_1+a_2+\overline{a_1}\,\overline{a_2}\lambda,
\label{eq:sigma1}
\\
\sigma_2
&=
a_1a_2+
(\overline{a_1}+\overline{a_2})\lambda,
\label{eq:sigma2}
\\
\sigma_3
&=
\lambda.
\label{eq:sigma3}
\end{align}
\end{theorem}

The importance of Theorem \ref{thm:sympara} is that, given the foci of the central conic, each elementary symmetric polynomial depends on a single unimodular complex parameter. This simple observation forms the basis of the unified \,approach developed in this paper.

We begin with a simple illustration involving the orthocenter. Readers may consult \cite{Hahn2019} for an introduction to the complex-number approach to Euclidean geometry.

\begin{theorem}[Orthocenter Criterion]
\label{thm:orthocentercriterion}
Let $\mathcal P$ be the family of triangles inscribed in the unit circle $\mathbb T$ and circumscribed about a central conic. Let $\triangle z_1z_2z_3\in\mathcal P$ and let
\[
z_H=z_1+z_2+z_3
\]
denote the complex coordinate of its orthocenter. Then the following statements are equivalent.
\begin{enumerate}
\item[(i)]
For every function $\Phi:\mathbb R_{>0}\to\mathbb R$, the quantity
\[
\Phi(|z_H|)
\]
remains invariant throughout the family $\mathcal P$.
\item[(ii)]
The center of $\mathbb T$ coincides either with the center of the conic or with one of its foci.
\end{enumerate}
\end{theorem}

\begin{proof}
By Theorem~\ref{thm:sympara},
\begin{equation}\label{eq:zHrelation}
z_H
=
a_1+a_2+\overline{a_1a_2}\lambda,
\qquad
\lambda\in\Lambda.
\end{equation}
Hence
\[
|z_H|^2
=
|a_1+a_2|^2
+
|a_1a_2|^2
+
2\operatorname{Re}\left((a_1+a_2)a_1a_2\,\overline{\lambda}\right).
\]
It follows that $|z_H|$ is independent of $\lambda$ if and only if either
\[
a_1+a_2=0
\]
or
\[
a_1a_2=0.
\]
These conditions are equivalent to the circumcenter of the triangle coinciding either with the center of the conic or with one of its foci.

Finally, every quantity of the form $\Phi(|z_H|)$ is invariant throughout $\mathcal P$ if and only if $|z_H|$ itself is invariant, since the identity function is obtained by taking $\Phi(t)=t$. This establishes the equivalence of (i) and (ii).
\end{proof}

For convenience, we shall say that a geometric quantity associated with a triangle in $\mathcal P$ satisfies the \emph{orthocenter criterion} if it is invariant throughout $\mathcal P$ precisely when the distance between the circumcenter and the orthocenter of the triangle is constant, that is, does not depend on the choice of the triangle in $\mathcal P$. Equivalently, this condition holds if and only if the circumcenter of the triangle coincides either with the center of the conic or with one of its foci.

The next corollary illustrates how Theorem \ref{thm:orthocentercriterion} immediately produces invariant quantities.

\begin{corollary}
Let
\[
F(z_1,z_2,z_3)
=
c_1\sigma_1\sigma_2+c_2\sigma_3,
\]
where $c_1,c_2\in\mathbb C$ are independent of the triangle. Then
\[
|F(z_1,z_2,z_3)|
\]
satisfies the orthocenter criterion.
\end{corollary}

\begin{proof}
From Theorem \ref{thm:sympara}, it follows that
\begin{equation}\label{eq:sigma1sigma2}
\sigma_1=z_H, \qquad    \sigma_2=\overline{z_H}\,\lambda,
\qquad
\sigma_3=\lambda.
\end{equation}
Hence
\[
F(z_1,z_2,z_3)
=
\left(c_1|z_H|^2+c_2\right)\lambda.
\]
Since $|\lambda|=1$,
\[
|F(z_1,z_2,z_3)|
=
\left|c_1|z_H|^2+c_2\right|.
\]
The conclusion follows immediately from Theorem \ref{thm:orthocentercriterion}.
\end{proof}

\begin{example}\label{ex:sympol}
Consider the symmetric polynomial
\[
F(z_1,z_2,z_3)
=
(z_1+z_2)(z_2+z_3)(z_3+z_1).
\]
Using the elementary symmetric polynomials, $F(z_1,z_2,z_3)$ can be written as 
\[
F(z_1,z_2,z_3)
=
\sigma_1\sigma_2-\sigma_3
\]
where $c_1=1, c_2=-1$ are independent of $z_1,z_2,z_3$. 

Therefore
\[
|F(z_1,z_2,z_3)|
=
\bigl||z_H|^2-1\bigr|.
\]
Consequently, $|F(z_1,z_2,z_3)|$ satisfies the orthocenter criterion. In other words, $|F(z_1,z_2,z_3)|$ \,remains invariant throughout the family $\mathcal P$ if and only if the circumcenter of the triangle coincides either with the center of the conic or with one of its foci.
\end{example}

\begin{remark}
    The modulus of the symmetric polynomial in Example \ref{ex:sympol} admits two natural \,geometric interpretations. Indeed,
\[
|F(z_1,z_2,z_3)|
=
8\,|OM_A|\,|OM_B|\,|OM_C|,
\]
where $M_A$, $M_B$, and $M_C$ are the midpoints of the sides of $\triangle z_1z_2z_3$, and
\[
|F(z_1,z_2,z_3)|
=
|AH|\,|BH|\,|CH|,
\]
where $H$ is the orthocenter of $\triangle z_1z_2z_3$. Thus the same symmetric polynomial can be used to study whether the product of distances from the circumcenter to the side midpoints, or a product of distances from the orthocenter to the vertices satisfy orthocenter criterion. See Subsection \ref{sec:orthictri}.
\end{remark}

\section{Geometric Applications}\label{sec:geomapp}

In this section we illustrate the effectiveness of the symmetric parametrization by deriving several geometric consequences for Poncelet triangles. Throughout, the principal tool is the orthocenter parametrization \eqref{eq:zHrelation}. As we shall see, many geometric properties follow immediately from this simple identity.

\subsection{The Orthocenter Locus}
\begin{theorem}[Orthocenter Locus]\label{thm:orthcirc}
Let a triangle be inscribed in $\mathbb T$ and circumscribed about a central conic with foci $a_1,a_2\in\mathbb C$. Then the orthocenter of the triangle lies on the circle centered at $a_1+a_2$ with radius $|a_1a_2|$.
\end{theorem}

\begin{figure}
  \begin{subfigure}[b]{0.4\textwidth}
    \centering
\begin{tikzpicture}[scale=2.1]
\clip(-1.2,-1.4) rectangle (1.2,1.2);
\draw [line width=1.pt] (0.,0.) circle (1.cm);
\draw [rotate around={8.130102354155984:(0.05,-0.45)},line width=1.pt,color=red] (0.05,-0.45) ellipse (0.4870318280275454cm and 0.33496268674563245cm);
\draw [line width=1.pt,color=blue] (-0.29701349243427755,0.954873282332265)-- (-0.4734897603533804,-0.8807993226839462);
\draw [line width=1.pt,color=blue] (-0.4734897603533804,-0.8807993226839462)-- (0.742165894161016,-0.6702162229789574);
\draw [line width=1.pt,color=blue] (0.742165894161016,-0.6702162229789574)-- (-0.29701349243427755,0.954873282332265);
\draw [line width=1.pt,color=zzttqq] (0.1,-0.9) circle (0.32984845004941277cm);
\draw [line width=1.pt,color=qqwuqq] (0.20015045541751703,-1.2142767670058732)-- (0.4010075633955934,-0.7651132075455577);
\draw [line width=1.pt,color=qqwuqq] (0.4010075633955934,-0.7651132075455577)-- (-0.21532202560606453,-0.803190805355646);
\draw [line width=1.pt,color=qqwuqq] (-0.21532202560606453,-0.803190805355646)-- (0.20015045541751703,-1.2142767670058732);
\begin{scriptsize}
\draw [fill=wewdxt] (0.,0.) circle (0.6pt);
\draw[color=wewdxt] (0.03514880746677769,0.07945618393525569) node {$O$};
\draw[color=black] (-0.66188243730359,0.8576486010419004) node {$\mathbb T$};
\draw [fill=black] (-0.3,-0.5) circle (0.6pt);
\draw[color=black] (-0.25368948231820343,-0.40989794338640745) node {$a_1$};
\draw [fill=black] (0.4,-0.4) circle (0.6pt);
\draw[color=black] (0.42,-0.3) node {$a_2$};
\draw[color=red] (-0.23220564258213047,-0.10196290716936088) node {$\mathcal D$};
\draw [fill=wewdxt] (-0.29701349243427755,0.954873282332265) circle (0.6pt);
\draw[color=wewdxt] (-0.3,1.05) node {$A$};
\draw [fill=wewdxt] (-0.4734897603533804,-0.8807993226839462) circle (0.6pt);
\draw[color=wewdxt] (-0.54,-0.96) node {$B$};
\draw [fill=wewdxt] (0.742165894161016,-0.6702162229789574) circle (0.6pt);
\draw[color=wewdxt] (0.8,-0.75) node {$C$};
\draw [fill=blue] (-0.02833735862664194,-0.5961422633306387) circle (0.6pt);
\draw[color=blue] (0,-0.5) node {$H$};
\draw[color=zzttqq] (0.2,-0.67) node {$\Gamma$};
\draw [fill=wewdxt] (0.20015045541751703,-1.2142767670058732) circle (0.6pt);
\draw[color=wewdxt] (0.20,-1.31) node {$D$};
\draw [fill=wewdxt] (0.4010075633955934,-0.7651132075455577) circle (0.6pt);
\draw[color=wewdxt] (0.48,-0.78) node {$E$};
\draw [fill=wewdxt] (-0.21532202560606453,-0.803190805355646) circle (0.6pt);
\draw[color=wewdxt] (-0.3,-0.78) node {$F$};
\draw [fill=qqwuqq] (0.185835993207046,-0.982580779907077) circle (0.6pt);
\draw[color=qqwuqq] (0.22,-0.88) node {$H'$};
\end{scriptsize}
\end{tikzpicture}
    \caption{$\mathcal D$ is an ellipse}
    \label{fig:orthcirc(A)}
    \end{subfigure}
  \begin{subfigure}[b]{0.45\textwidth}
    \centering
\definecolor{qqwuqq}{rgb}{0.,0.39215686274509803,0.}
\begin{tikzpicture}[scale=2.]
\clip(-1.5,-2) rectangle (3,2);
\draw [line width=1.pt] (0.,0.) circle (1.cm);
\draw [samples=50,domain=-0.99:0.99,rotate around={35.83765295427828:(0.6,-0.15)},xshift=0.6cm,yshift=-0.15cm,line width=1.pt,color=red] plot ({1.0636023712341884*(1+(\x)^2)/(1-(\x)^2)},{0.3181980450992955*2*(\x)/(1-(\x)^2)});
\draw [samples=50,domain=-0.99:0.99,rotate around={35.83765295427828:(0.6,-0.15)},xshift=0.6cm,yshift=-0.15cm,line width=1.pt,color=red] plot ({1.0636023712341884*(-1-(\x)^2)/(1-(\x)^2)},{0.3181980450992955*(-2)*(\x)/(1-(\x)^2)});
\draw [line width=1.pt,color=zzttqq] (1.2,-0.3) circle (1.3509256086106296cm);
\draw [line width=1.pt,color=qqwuqq] (0.7255680024340596,0.9648771796840954)-- (-0.10223656529286185,-0.6594160931486652);
\draw [line width=1.pt,color=qqwuqq] (-0.10223656529286185,-0.6594160931486652)-- (2.5506403813511196,-0.272241032733447);
\draw [line width=1.pt,color=qqwuqq] (2.5506403813511196,-0.272241032733447)-- (0.7255680024340596,0.9648771796840954);
\draw [line width=1.pt,color=blue] (-0.944776431597904,0.32771556920160333)-- (-0.12588168256050736,-0.9920452620701012);
\draw [line width=1.pt,color=blue] (-0.12588168256050736,-0.9920452620701012)-- (0.9200135880789069,0.3918864602792184);
\draw [line width=1.pt,color=blue] (0.9200135880789069,0.3918864602792184)-- (-0.944776431597904,0.32771556920160333);
\begin{scriptsize}
\draw [fill=wewdxt] (0.,0.) circle (0.7pt);
\draw[color=wewdxt] (0.04318848061106703,0.10748315747039414) node {$O$};
\draw[color=black] (-0.5995676432836226,0.9076489443596993) node {$\mathbb T$};
\draw [fill=black] (-0.3,-0.8) circle (0.7pt);
\draw[color=black] (-0.4,-0.81) node {$a_1$};
\draw [fill=black] (1.5,0.5) circle (0.7pt);
\draw[color=black] (1.59,0.6) node {$a_2$};
\draw[color=red] (1.899310756755936,1.5700812761287144) node {$\mathcal D$};
\draw [fill=wewdxt] (-0.944776431597904,0.32771556920160333) circle (0.7pt);
\draw[color=wewdxt] (-1.04,0.37) node {$A$};
\draw [fill=wewdxt] (-0.12588168256050736,-0.9920452620701012) circle (0.7pt);
\draw[color=wewdxt] (-0.15,-1.09) node {$B$};
\draw [fill=wewdxt] (0.9200135880789069,0.3918864602792184) circle (0.7pt);
\draw[color=wewdxt] (1,0.43) node {$C$};
\draw [fill=blue] (-0.1506445260795046,-0.2724432325892794) circle (0.7pt);
\draw[color=blue] (-0.25,-0.25) node {$H$};
\draw[color=zzttqq] (1.1,1.15) node {$\Gamma$};
\draw [fill=wewdxt] (0.7255680024340596,0.9648771796840954) circle (0.7pt);
\draw[color=wewdxt] (0.75,1.09) node {$D$};
\draw [fill=wewdxt] (-0.10223656529286185,-0.6594160931486652) circle (0.7pt);
\draw[color=wewdxt] (-0.19,-0.65) node {$E$};
\draw [fill=wewdxt] (2.5506403813511196,-0.272241032733447) circle (0.7pt);
\draw[color=wewdxt] (2.64,-0.27) node {$F$};
\draw [fill=qqwuqq] (0.773971818492316,0.6332200538019839) circle (0.7pt);
\draw[color=qqwuqq] (0.8367955315422653,0.7436805454069728) node {$H'$};
\end{scriptsize}
\end{tikzpicture}
    \caption{$\mathcal D$ is a hyperbola}
    \label{fig:orthcirc(B)}
    \end{subfigure}
        \caption{$\triangle ABC$ is inscribed in $\mathbb T$ and circumscribed about $\mathcal D$. The orthocenter $H$ of $\triangle ABC$ lies on the circle $\Gamma$ with center $a_1+a_2$ and radius $|a_1a_2|$. $(\Gamma,\mathcal D)$ is also a 3-Poncelet pair. The orthocenter $H'$ of the triangle $\triangle DEF$ inscribed in $\Gamma$ and circumscribed about $\mathcal D$ lies on $\mathbb T$. See Corollary \ref{cor:loc3pons}.}
    \label{fig:orthcirc}
\end{figure}

\begin{proof}
By \eqref{eq:zHrelation},
\[
z_H-(a_1+a_2)=\overline{a_1a_2}\lambda.
\]
Since $|\lambda|=1$, $z_H$ lies on the circle
\[
\Gamma: |z-(a_1+a_2)|=|a_1a_2|,
\]
which proves the theorem.
\end{proof}

The two most important special cases occur when the circumcenter of the triangle
coincides with either the center or a focus of the inscribed conic.

\begin{corollary}\label{cor:orthcirccen}
Let a triangle be inscribed in a circle and circumscribed about an ellipse. Assume that the circumcenter of the triangle coincides with the center of the ellipse. Then the orthocenter of the triangle lies on a circle concentric with the circumcircle of the triangle.
\end{corollary}

\begin{proof}
We may assume that $f \in \mathbb C$ and $-f\in \mathbb C$ be the foci of the conic. By Theorem
\ref{thm:orthcirc}, the locus circle $\Gamma$ reduces to
\[
|z|=|f|^2,
\]
which proves the claim.
\end{proof}

\begin{corollary}\label{cor:loc3pons}
Let $\mathcal P$ be the family of triangles inscribed in a circle $\mathcal C$ and circumscribed about a central conic $\mathcal D$. Let $\Gamma$ denote the locus  of the orthocenters of the triangles in $\mathcal P$. Assume that the center of $\mathcal C$ is distinct from the foci of $\mathcal D$. Then  $(\Gamma, \mathcal D)$ forms a 3-Poncelet pair if and only if $(\mathcal C,\mathcal D)$ forms a 3-Poncelet pair.

Moreover, the orthocenters of the triangles inscribed in $\Gamma$ and circumscribed about $\mathcal D$ lies on $\mathcal C$. 
\end{corollary}

\begin{proof}
Without loss of generality, assume that $\mathcal C=\mathbb T$, and let
$a_1,a_2$ denote the foci of $\mathcal D$. Since the circumcenter of
$\mathbb T$ is distinct from the foci of $\mathcal D$, we have
$a_1a_2\neq0$.

Suppose first that $(\mathbb T,\mathcal D)$ is a $3$-Poncelet pair.
By Theorem~\ref{thm:orthcirc}, the locus of the orthocenters of the
triangles in $\mathcal P$ is the circle
\[
\Gamma:
|z-(a_1+a_2)|=|a_1a_2|.
\]

Consider the affine transformation $T:\mathbb C\to\mathbb C$, defined by
\[
w=T(z)=\frac{z}{a_1a_2}-\frac{a_1+a_2}{a_1a_2}.
\]
It maps $\Gamma$ onto the unit circle and transforms $\mathcal D$ into the
central conic
\[
T(\mathcal D):
\left|
\left|w+\frac1{a_1}\right|
\pm
\left|w+\frac1{a_2}\right|
\right|
=
\left|
1-\frac1{\overline{a_1}a_2}
\right|,
\]
where the positive sign corresponds to an ellipse and the negative sign to a
hyperbola.

By \cite[Theorem~2.1]{Murad2026b}, the pair
$(T(\Gamma),T(\mathcal D))$ is a $3$-Poncelet pair. Since affine
transformations preserve the Poncelet property, it follows that
$(\Gamma,\mathcal D)$ is also a $3$-Poncelet pair. The converse follows by
applying the same argument to the inverse affine transformation $T^{-1}$.

Finally, let $w_H$ be the orthocenter of a triangle inscribed in $\Gamma$
and circumscribed about $\mathcal D$. Since $(T(\Gamma),T(\mathcal D))$
is a $3$-Poncelet pair with circumcircle $\mathbb T$, Theorem~\ref{thm:orthcirc}
shows that
\[
\left|
w_H+\frac1{a_1}+\frac1{a_2}
\right|
=
\left|
\frac1{a_1a_2}
\right|.
\]
Applying the inverse transformation
\[
z=T^{-1}(w)=a_1a_2\,w+(a_1+a_2)
\]
yields
\[
|z|=1.
\]
Hence the orthocenters of the triangles inscribed in $\Gamma$ and
circumscribed about $\mathcal D$ lie on the original circumcircle
$\mathcal C$.
\end{proof}

\begin{corollary}\label{cor:focusorthocenter}
Let a triangle be circumscribed about a central conic. Then the circumcenter of the triangle coincides with one focus of the conic if and only if the orthocenter coincides with the other
focus.
\end{corollary}

\begin{proof}
If $a_2=0$, then
\[
z_H=a_1.
\]
Conversely, assume $z_H=a_1$. Equation
\eqref{eq:zHrelation} gives
\[
a_2+\overline{a_1}\,\overline{a_2}\lambda=0.
\]
If $a_2\neq0$, then
\[
\overline{a_1}\lambda\frac{\overline{a_2}}{a_2}=-1,
\]
whose modulus yields
\[
|a_1|=1,
\]
contrary to the fact that a focus of a nondegenerate central conic
cannot lie on the circumcircle. Hence $a_2=0$.
\end{proof}

\begin{corollary}\label{cor:eccen}
Let a triangle be inscribed in a unit circle and circumscribed about a central conic. Assume that the circumcenter of the triangle coincides with one of the foci of the conic. Then the eccentricity of the conic equals the distance from the circumcenter to the other focus. Equivalently,
\[
e=|OH|,
\]
where $e$ denotes the eccentricity, and $O$ and $H$ denote the circumcenter and orthocenter of the triangle, respectively.
\end{corollary}

\begin{proof}
The eccentricity of the conic with foci $a_1,a_2 \in \mathbb C$ inscribed in a triangle inscribed in $\mathbb T$ is given by (\cite[Corollary 1]{Murad2026b})
\[
e=\frac{|a_1-a_2|}{1-\overline{a_1}a_2}.
\]
Assuming $a_2=0$ and using $z_H=a_1$ (Corollary \ref{cor:focusorthocenter}), we find
\[
e=|a_1|=|z_H|.
\]
This completes the proof.
\end{proof}

\subsection{Families with Constant Orthocenter Modulus}

The orthocenter locus immediately yields information about the
possible types of triangles occurring in a particular Poncelet family.

\begin{proposition}\label{prop:obltri}
The Poncelet family $\mathcal P$ contains at most two right triangles.

Moreover, if the circumcenter of a triangle in $\mathcal P$ coincides either with the center of the conic or with one of its foci, then the triangle is oblique.
\end{proposition}

\begin{proof}
A triangle is right if and only if its orthocenter lies on its circumcircle. Since the orthocenter moves on a circle, there are at most two intersections with the circumcircle.

Assume that the triangles in $\mathcal P$ are inscribed in $\mathbb T$ and circumscribed about a central conic $\mathcal D$ with foci $a_1,a_2\in \mathbb C$. Suppose $\triangle ABC \in \mathcal P$. By Theorem \ref{thm:orthcirc},
\begin{itemize}
    \item if the circumcenter of $\triangle ABC$ coincides with the center of $\mathcal D$, then
\[
|z_H|=|a_1|^2;
\]
in this case, $|z_H| \neq 1$, since
$|a_1|\neq 1$;
    \item if the circumcenter of $\triangle ABC$ coincides with one focus of $\mathcal D$, say $a_2$, then 
\[
|z_H|=|a_1|;
\]
in this case, $|z_H| \neq 1$, since $|a_1|<1$ if $\mathcal D$ is an ellipse and $|a_1|>1$ if $\mathcal D$ is a hyperbola.
\end{itemize}
In either case, $|z_H| \neq 1$, which proves the claim.
\end{proof}

The preceding proposition immediately gives the following geometric
characterizations.

\begin{corollary}
Let a triangle be circumscribed about an ellipse. Assume that the circumcenter of the triangle coincides with the center of the ellipse. Then the triangle is acute if the foci of the ellipse lie inside the circumcircle of the triangle and obtuse if the foci lie outside the circumcircle.
\end{corollary}

\begin{proof}
Since, from Proposition \ref{prop:obltri},
\[
|z_H|=|a_1|^2,
\]
the orthocenter lies inside the circumcircle precisely when
$|a_1|<1$, and outside when $|a_1|>1$. The conclusion follows from
the classical characterization of acute and obtuse triangles.
\end{proof}

\begin{corollary}\label{cor:focus}
Let a triangle be circumscribed about a central conic having the circumcenter of the triangle as one of its foci. Then the triangle is acute if the conic is an ellipse and obtuse if it is a hyperbola.
\end{corollary}

\begin{proof}
Since, from Proposition \ref{prop:obltri}
\[
z_H=a_1,
\]
the conclusion follows immediately from $|a_1|<1$ for ellipses and
$|a_1|>1$ for hyperbolas.
\end{proof}

\subsection{Bounds on the Focal Parameters}
The orthocenter parametrization also yields several \,relations between the focal parameters and the geometry of the conic.

The following identity (see Corollary \ref{cor:sumsqdis})
\[
|z_1-z_2|^2+|z_2-z_3|^2+|z_3-z_1|^2
=
9-|z_H|^2
\]
immediately yields bounds on the focal parameters.

\begin{proposition}\label{prop:bounds}
Let a triangle be inscribed in the unit circle and circumscribed about
a central conic.

\begin{enumerate}
\item[(a)]
If the circumcenter coincides with the center of the conic, then
\[
|f|<\sqrt3.
\]
where $f\in \mathbb C$ is a focus of the conic.
\item[(b)]
If the circumcenter coincides with one focus of the conic, then
\[
|f|<3,
\]
where $f\in \mathbb C$ is the other focus of the conic.
\end{enumerate}
\end{proposition}

\begin{proof}
Since
\[
9-|z_H|^2>0,
\]
we have
\[
|z_H|<3.
\]

In case (a),
\[
|z_H|=|f|^2,
\]
which gives
\[
|f|<\sqrt3.
\]

In case (b),
\[
z_H=f,
\]
and therefore
\[
|f|<3.
\]
\end{proof}

\begin{corollary}
Let a triangle be circumscribed about a hyperbola $\mathcal D$. Assume that the circumcenter of the triangle coincides with one of the foci of $\mathcal D$. Then the eccentricity $e$ of $\mathcal D$ satisfies
\[
1<e<3.
\]
\end{corollary}

\begin{proof}
Since $e=|f|$ (Corollary \ref{cor:eccen}), the result follows
immediately from Proposition \ref{prop:bounds}.
\end{proof}

We conclude with one final application concerning the congruence between the circumcircle of a triangle circumscribed about a central conic and the circle traced by its orthocenter.

\begin{proposition}
Let $\Gamma$ denote the locus of the orthocenters of the triangles in the Poncelet family $\mathcal P$. If $\Gamma$ is congruent to the common circumcircle of the triangles in $\mathcal P$, then the associated central conic is necessarily a hyperbola.
\end{proposition}

\begin{proof}
The congruence assumption gives
\[
|a_1a_2|=1.
\]
Hence
\[
|a_1||a_2|=1,
\]
so one focus lies inside the unit circle if and only if the other lies
outside. Therefore the conic is a hyperbola. See \cite[Theorem 2.2]{Dragovic-Murad2026}.
\end{proof}

\section{Some Invariants from Symmetric Parametrization}\label{sec:invsympar}

In \cite{Dragovic-Murad2026}, we proved that if a triangle $\triangle ABC$ is circumscribed about a central conic, then the sum of the squared side lengths,
\[
|AB|^{2}+|BC|^{2}+|CA|^{2},
\]
remains invariant throughout the associated Poncelet family if and only if the circumcenter of $\triangle ABC$ coincides either with the center of the conic or with one of its foci. Equivalently, the sum satisfies the orthocenter criterion.

In this section we show that many other geometric quantities associated with the triangles in the Poncelet family, such as squared distances, angles, ratio of areas, and circles associated with orthic triangles, tangential triangles, polar circles satisfy the orthocenter criterion as well.

\subsection{Length Invariants}\label{sec:lenghtinv}

\begin{lemma}
Let $z_1,z_2,z_3\in\mathbb C$. Then
\begin{equation}\label{eq:z1z2z3id}
|z_1-z_2|^2+|z_2-z_3|^2+|z_3-z_1|^2
=
3\bigl(|z_1|^2+|z_2|^2+|z_3|^2\bigr)
-|z_1+z_2+z_3|^2.
\end{equation}
\end{lemma}

\begin{proof}
Summing the identity
\[
|z_i-z_j|^2
=
|z_i|^2+|z_j|^2
-2\operatorname{Re}(\overline{z_i} z_j)
\]
for the pair of indices $(i,j)=(1,2)$, $(2,3)$ and $(3,1)$ gives
\begin{equation}\label{eq:z1z2z3expans}
|z_1-z_2|^2+|z_2-z_3|^2+|z_3-z_1|^2
=
2(|z_1|^2+|z_2|^2+|z_3|^2)
-
2\operatorname{Re}
\left(
\overline{z_1}z_2+\overline{z_2}z_3+\overline{z_3}z_1
\right).
\end{equation}
On the other hand,
\begin{equation}\label{eq:z1z2bar}
|z_1+z_2+z_3|^2
=
|z_1|^2+|z_2|^2+|z_3|^2
+
2\operatorname{Re}
\left(
\overline{z_1}z_2+\overline{z_2}z_3+\overline{z_3}z_1
\right).
\end{equation}
Adding \eqref{eq:z1z2z3expans} and \eqref{eq:z1z2bar} yields
\[
|z_1-z_2|^2+|z_2-z_3|^2+|z_3-z_1|^2
+
|z_1+z_2+z_3|^2
=
3(|z_1|^2+|z_2|^2+|z_3|^2),
\]
which is equivalent to \eqref{eq:z1z2z3id}.
\end{proof}

\begin{corollary}\label{cor:sumsqdis}
If $z_1,z_2,z_3\in\mathbb T$, then
\begin{align}
|z_1-z_2|^2+|z_2-z_3|^2+|z_3-z_1|^2
&=9-|z_H|^2,\label{eq:sumsq}\\
|z_1+z_2|^2+|z_2+z_3|^2+|z_3+z_1|^2
&=3+|z_H|^2.\label{eq:possumsq}
\end{align}
\end{corollary}

\begin{proof}
Since $|z_1|=|z_2|=|z_3|=1$ and $z_H=z_1+z_2+z_3$, identity \eqref{eq:z1z2z3id} immediately gives \eqref{eq:sumsq}. Moreover,
\[
|z_1+z_2|^2+|z_2+z_3|^2+|z_3+z_1|^2
+
|z_1-z_2|^2+|z_2-z_3|^2+|z_3-z_1|^2
=12,
\]
which together with \eqref{eq:sumsq} proves \eqref{eq:possumsq}.
\end{proof}

The preceding corollary immediately yields an alternative proof of the invariance of the sum of the squared lengths established in \cite[Theorem 3.5]{Dragovic-Murad2026}.

\begin{proposition}\label{prop:abcoh}
Let $l\in\mathbb R$ and let $z_1,z_2,z_3\in\mathbb C$. Then
\begin{equation}\label{eq:zlzh}
|z_1+lz_H|^2+|z_2+lz_H|^2+|z_3+lz_H|^2
=
|z_1|^2+|z_2|^2+|z_3|^2
+l(3l+2)|z_H|^2.
\end{equation}
\end{proposition}

\begin{proof}
For each $k=1,2,3$,
\[
|z_k+lz_H|^2
=
|z_k|^2+l^2|z_H|^2
+
2l\operatorname{Re}(\overline{z_k}z_H).
\]
Summing over $k$ gives
\[
\sum_{k=1}^3|z_k+lz_H|^2
=
\sum_{k=1}^3|z_k|^2
+
3l^2|z_H|^2
+
2l\operatorname{Re}\left(\overline{z_H}z_H\right).
\]
Since
\[
\operatorname{Re}(\overline{z_H}z_H)
=
|z_H|^2,
\]
the last sum is equivalent to \eqref{eq:zlzh}.
\end{proof}

\begin{corollary}\label{cor:mediansqsum}
Let $z_1,z_2,z_3\in\mathbb T$. Then
\begin{equation}\label{eq:mediansqsum}
\left|z_1-\frac{z_2+z_3}{2}\right|^2
+
\left|z_2-\frac{z_3+z_1}{2}\right|^2
+
\left|z_3-\frac{z_1+z_2}{2}\right|^2
=
\frac{27}{4}-\frac34|z_H|^2.
\end{equation}
\end{corollary}

\begin{proof}
Since
\[
z_1-\frac{z_2+z_3}{2}
=
\frac32\left(z_1-\frac{z_H}{3}\right),
\]
we obtain
\[
\left|z_1-\frac{z_2+z_3}{2}\right|^2
=
\frac94
\left|z_1-\frac{z_H}{3}\right|^2.
\]
Summing cyclically and applying Proposition \ref{prop:abcoh} with $l=-\frac13$ yields \eqref{eq:mediansqsum}.
\end{proof}

\begin{corollary}
The sum of the squares of the median lengths of a triangle in the family $\mathcal P$ satisfies the orthocenter criterion.
\end{corollary}

\begin{proof}
The median from the vertex $z_1$ has length
\[
\left|z_1-\frac{z_2+z_3}{2}\right|,
\]
and similarly for the other two vertices. The result therefore follows immediately from Corollary \ref{cor:mediansqsum} together with Theorem \ref{thm:orthocentercriterion}.
\end{proof}

\begin{corollary}
The sum of the areas of the regular convex polygons on the sides of a triangle in $\mathcal P$ satisfies the orthocenter criterion.
\end{corollary}
\begin{proof} Let $\triangle ABC \in \mathcal P$. For each vertex $X \in \{A,B,C\}$, let $P_X$ denote the regular convex $n$-gon on the side containing the remaining two vertices of $\triangle ABC$. Then
\[
\operatorname{Area}(P_A)+\operatorname{Area}(P_B)+\operatorname{Area}(P_C)=\frac{1}{4}n\left(|AB|^2+|BC|^2+|CA|^2\right)\cot\left(\frac{\pi}{n}\right)
\]
In particular, the sum of the areas of the outer (equivalently, inner) Napoleon triangles of $\triangle ABC$ satisfies the orthocenter criterion. For further background on Napoleon triangles, see, for example, \cite{Hahn2019}.
\end{proof}

\subsection{Angular Invariants}\label{sec:anginv}

The identities established in the previous subsection immediately yield several invariant angular quantities. We begin with a characterization involving the sums of squared sines and the product of the cosines of the angles.

\begin{theorem}\label{thm:sina}
Let $\triangle ABC \in \mathcal P$. Then the sum
\[
\sin^2A+\sin^2B+\sin^2C
\]
and the product
\[
\cos A\cos B\cos C
\]
satisfy the orthocenter criterion.
\end{theorem}

\begin{proof}
We may assume that the common circumcircle of $\mathcal P$ is the unit circle $\mathbb T$. By the law of sines,
\[
\frac{|BC|}{\sin A}
=
\frac{|CA|}{\sin B}
=
\frac{|AB|}{\sin C}
=
2,
\]
and hence
\[
\sin^2A+\sin^2B+\sin^2C
=
\frac{|AB|^2+|BC|^2+|CA|^2}{4}.
\]
Let $z_1,z_2,z_3 \in \mathbb T$ be the complex coordinates of the vertices $A,B,C$, respectively. By Corollary \ref{cor:sumsqdis},
\[
|AB|^2+|BC|^2+|CA|^2
=
9-|z_H|^2,
\]
so that
\[
\sin^2A+\sin^2B+\sin^2C
=
\frac{9-|z_H|^2}{4}.
\]
The trigonometric identity 
\[
\sin^2A+\sin^2B+\sin^2C
=
2+2\cos A\cos B\cos C
\]
for triangles gives
\[
\cos A \cos B \cos C
=
\frac{1-|z_H|^2}{8}.
\]
Therefore both quantities satisfy the orthocenter criterion.
\end{proof}

\begin{corollary}
Let $\triangle ABC$ be inscribed in $\mathbb T$ and circumscribed about a central conic. Assume that the circumcenter of $\triangle ABC$ coincides with one of the foci of the conic. Then
\begin{align*}
    \sin^2A+\sin^2B+\sin^2C&=\frac{9-e^2}{4}\\
    \cos A \cos B \cos C&=\frac{1-e^2}{8}.
\end{align*}
where $e$ is the eccentricity of the conic.
\end{corollary}

\begin{proof}
The result follows immediately from Theorem \ref{thm:sina} together with Corollary \ref{cor:eccen}.
\end{proof}

The next result provides another example of angular invariants associated with the family $\mathcal P$.

\begin{theorem}\label{thm:focanginv}
Let $\triangle ABC \in \mathcal P$ and $\ell$ be a line through the circumcenter $O$ of $\triangle ABC$. Let
\[
\alpha=\measuredangle(\ell,OA),\qquad
\beta=\measuredangle(\ell,OB),\qquad
\gamma=\measuredangle(\ell,OC)
\]
denote the corresponding angles measured counterclockwise. Then the quantity
\[
\cos(\alpha-\beta)+\cos(\beta-\gamma)+\cos(\gamma-\alpha)
\]
satisfy the orthocenter criterion.
\end{theorem}

\begin{proof}
We may assume that $\triangle ABC$ is inscribed in $\mathbb T$. Choose coordinates so that $\ell$ is the real axis and the perpendicular through $O$ is the imaginary axis. 

Let $z_1,z_2,z_3$ and $z_H$ denote the complex coordinates of the vertices and the orthocenter of $\triangle ABC$, respectively. Then
\[
z_1=e^{i\alpha},\qquad
z_2=e^{i\beta},\qquad
z_3=e^{i\gamma}.
\]
By \eqref{eq:z1z2bar},
\[
|z_1+z_2+z_3|^2
=
3+
2\operatorname{Re}\left(
\overline z_1z_2+
\overline z_2z_3+
\overline z_3z_1
\right).
\]
Since
\[
\operatorname{Re}(\overline z_1z_2)=\cos(\beta-\alpha),\quad
\operatorname{Re}(\overline z_2z_3)=\cos(\gamma-\beta),\quad
\operatorname{Re}(\overline z_3z_1)=\cos(\alpha-\gamma),
\]
we obtain
\[
\cos(\alpha-\beta)+
\cos(\beta-\gamma)+
\cos(\gamma-\alpha)
=
\frac{|z_H|^2-3}{2}.
\]
This proves the claim. 
\end{proof}

\begin{corollary}
    Let $\triangle ABC$ be inscribed in $\mathbb T$ and circumscribed about a central conic $\mathcal D$. Let $\ell$ be the major axis of $\mathcal D$ and $\alpha,\beta,\gamma$ be the angles defined as in Theorem \ref{thm:focanginv}. 
\begin{itemize}
    \item[(a)] If the circumcenter of $\triangle ABC$ coincides with the center of $\mathcal D$, then
\begin{equation}\label{eq:alphabetgammcen}
    \cos(\alpha-\beta)+\cos(\beta-\gamma)+\cos(\gamma-\alpha)
=
\frac{|f|^4-3}{2},
\end{equation}
where $f \in \mathbb C$ is either focus of $\mathcal D$.
\item[(b)] If the circumcenter of $\triangle ABC$ coincides with one of the foci of $\mathcal D$, then
\begin{equation}\label{eq:alphabetgammfoc}
    \cos(\alpha-\beta)+\cos(\beta-\gamma)+\cos(\gamma-\alpha)
=
\frac{|f|^2-3}{2},
\end{equation} 
where $f \in \mathbb C$ is the focus of $\mathcal D$ distinct from $O$.
\end{itemize}
\end{corollary}
\begin{proof}
If $\ell$ is the major axis of $\mathcal D$, then  
\[
|z_H|=|f|^2
\]
when the circumcenter of $\triangle ABC$ coincides with the center of $\mathcal D$, and
\[
|z_H|=|f|
\]
when the circumcenter of $\triangle ABC$ coincides with one of the foci of $\mathcal D$ (cf. Theorem \ref{thm:orthcirc} and \,Corollary \ref{cor:orthcirccen}). 

These prove \eqref{eq:alphabetgammcen} and \eqref{eq:alphabetgammfoc}, respectively.    
\end{proof}
\subsection{Orthic Triangles}\label{sec:orthictri}

Throughout this subsection, $H$ denotes the orthocenter of a triangle $\triangle ABC$, and
$H_A$, $H_B$, and $H_C$ denote the feet of the altitudes from $A$, $B$, and $C$, respectively.

\begin{proposition}\label{prop:orthiccoord}
Let $\triangle z_1z_2z_3$ be inscribed in the unit circle $\mathbb T$, and let $w_1,w_2,w_3\in \mathbb C$ be the feet of the altitudes from $z_1,z_2,z_3$ respectively. If $z_H=z_1+z_2+z_3$, then
\[
w_k
=
\frac12\left(z_H-\overline{z_k}z_{k+1}z_{k+2}\right), \qquad k=1,2,3
\]
where the indices are taken modulo $3$.
\end{proposition}

\begin{proof}
Since $w_1$ lies on the side $[z_2,z_3]$, there exists a real
parameter $t$ such that
\begin{equation}\label{eq:wtz}
    w_1=(1-t)z_2+tz_3.
\end{equation}
Moreover, $[z_1,w_1]$ is perpendicular to $[z_2,z_3]$, so
\[
\frac{w_1-z_1}{z_3-z_2}\in i\mathbb R.
\]
Equivalently,
\[
\frac{w_1-z_1}{z_3-z_2}
=
-\frac{\overline{w_1}-\overline{z_1}}
{\overline{z_3}-\overline{z_2}}.
\]
Using $|z_k|=1$, so that $\overline{z_k}=1/z_k$, and solving together with the linear relation \eqref{eq:wtz}, we obtain
\[
w_1=\frac12\left(z_1+z_2+z_3-\overline{z_1}z_2z_3\right).
\]
Since $z_H=z_1+z_2+z_3$, this becomes
\[
w_1=\frac12\left(z_H-\overline{z_1}z_2z_3\right).
\]
The remaining formulas follow by cyclic permutation of the indices.
\end{proof}

\begin{theorem}\label{thm:AH_HHA}
Let $\triangle ABC\in\mathcal P$, and let $H$ denote its orthocenter. Then
\[
|AH|\,|HH_A|
=
|BH|\,|HH_B|
=
|CH|\,|HH_C|.
\]
Moreover, this common value satisfies the orthocenter criterion.
\end{theorem}

\begin{proof}
We may assume that $\triangle ABC$ is inscribed in the unit circle $\mathbb T$. Let the complex coordinates of the vertices be $z_1,z_2,z_3$. Then
\[
z_H-z_1=z_2+z_3,
\]
so that
\[
|AH|=|z_2+z_3|.
\]
By Proposition \ref{prop:orthiccoord},
\[
w_1
=
\frac12\left(z_H-\overline{z_1}z_2z_3\right).
\]
Hence
\[
z_H-w_1
=
\frac12\left(z_H+\overline{z_1}z_2z_3\right).
\]
Since $|z_1|=1$, we have $\overline{z_1}=1/z_1$, and therefore
\[
z_H+\overline{z_1}z_2z_3
=
\frac{(z_1+z_2)(z_1+z_3)}{z_1}.
\]
Taking absolute values gives
\[
|HH_A|
=
\frac12|z_1+z_2|\,|z_1+z_3|.
\]
Consequently,
\[
|AH|\,|HH_A|
=
\frac12
|z_1+z_2|
|z_2+z_3|
|z_3+z_1|.
\]
Similarly,
\[
|BH|\,|HH_B|=|CH|\,|HH_C|=
\frac12
|z_1+z_2|
|z_2+z_3|
|z_3+z_1|.
\] 
By Example \ref{ex:sympol},
\[
|z_1+z_2|
|z_2+z_3|
|z_3+z_1|
=
\bigl||z_H|^2-1\bigr|,
\]
and the conclusion follows from Theorem \ref{thm:orthocentercriterion}.
\end{proof}

\begin{theorem}
Let the triangle $\triangle z_1z_2z_3$ be inscribed in $\mathbb T$ and circumscribed about a central conic. Then the centroid of the orthic triangle of $\triangle z_1z_2z_3$ lies on the circle
\[
\left|z-\frac56z_H\right|
=
\frac{|z_H|^2}{6},
\]
where $z_H$ denotes the complex coordinate of the orthocenter of $\triangle z_1z_2z_3$.
\end{theorem}

\begin{proof}
Let $w_k$ be the complex coordinate of the foot of altitude from $z_k$, $k=1,2,3$. 

By Proposition \ref{prop:orthiccoord},
\[
w_k
=
\frac12
\left(
z_H-\overline{z_k}z_{k+1}z_{k+2}
\right),
\]
where the indices are taken modulo $3$. Therefore,
\[
w_1+w_2+w_3
=
\frac12
\left(
3z_H-
\left(
\overline{z_1}z_2z_3
+\overline{z_2}z_3z_1
+\overline{z_3}z_1z_2
\right)
\right).
\]

Let $\lambda=z_1z_2z_3$. Then
\[
\overline{z_1}z_2z_3
+\overline{z_2}z_3z_1
+\overline{z_3}z_1z_2
=
\left(\overline{z_1}^{\,2}
+\overline{z_2}^{\,2}
+\overline{z_3}^{\,2}\right) \lambda.
\]
Using the identity
\[
\overline{z_1}^{\,2}
+\overline{z_2}^{\,2}
+\overline{z_3}^{\,2}
=
\overline{\sigma_1}^{\,2}
-
2\overline{\sigma_2}
=
\overline{z_H}^{\,2}
-
2z_H\overline{\lambda},
\]
we obtain
\[
\overline{z_1}z_2z_3
+\overline{z_2}z_3z_1
+\overline{z_3}z_1z_2
=
\overline{z_H}^{\,2}\lambda
-
2z_H.
\]
Hence the centroid $w$ of the orthic triangle satisfies
\[
w
=
\frac{w_1+w_2+w_3}{3}
=
\frac16
\left(
5z_H-\overline{z_H}^{\,2}\lambda
\right).
\]
Hence,
\[
w-\frac56z_H
=
-\frac16\overline{z_H}^{\,2}\lambda.
\]
Since $|\lambda|=1$, it follows that
\[
\left|w-\frac56z_H\right|
=
\frac16|\overline{z_H}|^2
=
\frac16|z_H|^2,
\]
which proves the theorem.
\end{proof}

In \cite{Dragovic-Murad2026}, we established that the area of a Poncelet triangle is invariant precisely when the associated central conic is a circle. The next theorem shows an invariance for the ratio of the areas of a triangle and its orthic triangle. This invariance is characterized by the same orthocenter criterion developed in Theorem \ref{thm:orthocentercriterion}.

\begin{theorem}\label{thm:orthicarearatio}
The ratio of the area of a triangle in $\mathcal P$ to the area of its orthic triangle satisfies the orthocenter criterion.
\end{theorem}
\begin{proof} Let $\triangle z_1z_2z_3\in\mathcal P$, and let $w_1,w_2,w_3$ denote the feet of altitudes from $z_1,z_2,z_3$.

By Proposition \ref{prop:orthiccoord}
\[
w_k=\frac{1}{2}(z_H-\overline{z_k}z_{k+1}z_{k+2})
\]
where the indices are taken modulo 3. This gives
    \begin{align*}
    |w_1-w_2|=\frac{1}{2}|z_1+z_2||z_1-z_2|.
\end{align*}
The remaining two identities follow cyclically.

This gives
\begin{align*}
    \frac{\operatorname{Area}(\triangle H_AH_BH_C)}{\operatorname{Area}(\triangle ABC)}&=2\frac{|w_1-w_2||w_2-w_3||w_3-w_1|}{|z_1-z_2||z_2-z_3||z_3-z_1|}\\
    &=\frac{1}{4}|z_1+z_2||z_2+z_3||z_3+z_1|\\
    &=\frac{1}{4} \left||z_H|^2-1\right|.
\end{align*}
Now Example \ref{ex:sympol} and Theorem \ref{thm:orthocentercriterion} complete the proof.
\end{proof}

\begin{corollary}
Let $\triangle ABC$ be circumscribed about an ellipse $\mathcal{D}$. Assume that the circumcenter of $\triangle ABC$ coincides with the center of $\mathcal D$. Then
    \begin{equation*}
    \frac{\mathrm{Area}(\triangle ABC)}{\mathrm{Area}(\triangle H_A H_B H_C)}=\frac{\mathrm{Area}(\mathcal{R}_S)}{\mathrm{Area}(\mathcal{R}_R)}=\frac{\mathrm{Area}(\mathcal{C})}{\mathrm{Area}(\mathcal{D})}.
\end{equation*}
where $\mathcal{R}_S$ and $\mathcal{R}_R$ are the circumscribed square of the circumcircle of $\triangle ABC$ and the circumscribed rectangle of the ellipse $\mathcal{D}$.
\end{corollary}
\begin{proof} Let $a$ and $b$ denote the lengths of the semi-axes of the ellipse. It follows from \cite[Corollary 2.2]{Murad2026b} that
\[
ab=\frac14\left(|z_H|^2-1\right).
\]
Hence, Theorem \ref{thm:orthicarearatio} gives
    \begin{equation*}
    \frac{\mathrm{Area}(\triangle ABC)}{\mathrm{Area}(\triangle H_A H_B H_C)}=\frac{1}{ab}=\frac{\mathrm{Area}(\mathcal{R}_S)}{\mathrm{Area}(\mathcal{R}_R)}.
\end{equation*}
The same ratio can also be written as 
    \begin{equation*}
    \frac{\mathrm{Area}(\triangle ABC)}{\mathrm{Area}(\triangle H_A H_B H_C)}=\frac{\pi}{\pi ab}=\frac{\mathrm{Area}(\mathcal{C})}{\mathrm{Area}(\mathcal{D})}.
    \end{equation*}
This completes the proof.
\end{proof}

\begin{corollary}
Let $\triangle ABC$ be circumscribed about a central conic. Assume that the circumcenter of $\triangle ABC$ coincides with one of the foci of the conic. Then
    \begin{equation*}
    \frac{\mathrm{Area}(\triangle ABC)}{\mathrm{Area}(\triangle H_A H_B H_C)}=4\frac{a^2}{b^2}
\end{equation*}
where $a$ and $b$ $(a>b>0)$ are the semi-axes of the conic.
\end{corollary}

\begin{theorem}\label{thm:alphah}
Let $\triangle ABC$ be an oblique triangle. Assume that $\angle A$ is obtuse whenever $\triangle ABC$ is obtuse. Set
\begin{equation*}
    \alpha_H=\angle H_CH_AH_B,\qquad 
\beta_H=\angle H_AH_BH_C,\qquad 
\gamma_H=\angle H_BH_CH_A.
\end{equation*}
Then
\begin{equation*}
    \varepsilon\cos \alpha_H+\cos \beta_H+\cos \gamma_H=\varepsilon(-3+2(\sin^2 A+\sin^2 B+\sin^2 C)),
\end{equation*}
with $\varepsilon=1$ (respectively, $-1$) if $\triangle ABC$ is acute (respectively, obtuse).

Consequently,
\[
\varepsilon\cos \alpha_H+\cos \beta_H+\cos \gamma_H
\]
satisfies the orthocenter criterion.
\end{theorem}

\begin{proof} It follows from the property of orthic triangle $\triangle H_AH_BH_C$ of $\triangle ABC$ that
    \begin{equation*}
        \alpha_H=\pi-2\angle A,\qquad \beta_H=\pi-2\angle B\qquad \gamma_H=\pi-2\angle C
    \end{equation*}
if $\triangle ABC$ is acute; and
\begin{equation*}
        -\alpha_H=\pi-2\angle A,\qquad \beta_H=2\angle B\qquad \gamma_H=2\angle C
    \end{equation*}
    if $\triangle ABC$ is obtuse with obtuse angle $\angle A$.

Now, to prove (a), we first assume $\triangle ABC$ is acute. Then we have
\begin{align*}
    \cos \alpha_H+\cos \beta_H+\cos \gamma_H&=-\cos (2A)-\cos (2B )-\cos (2C)\\
    &=-3+2(\sin^2 A+\sin^2 B+\sin^2 C).
\end{align*}

Similarly, if $\triangle ABC$ is obtuse with obtuse angle $\angle A$, then we have
\begin{align*}
    -\cos \alpha_H+\cos \beta_H+\cos \gamma_H&=3-2(\sin^2 A+\sin^2 B+\sin^2 C).
\end{align*}
The conclusion follows from Theorem \ref{thm:sina}.
\end{proof}

\begin{theorem}\label{thm:congorth}
Let $\triangle ABC$ be inscribed in a circle and circumscribed about an ellipse. Assume that the foci of the ellipse lie inside the circumcircle of $\triangle ABC$. Let $\triangle H_AH_BH_C$ denote the orthic triangle of $\triangle ABC$. Then the area of the incircle of $\triangle H_AH_BH_C$ satisfies the orthocenter criterion.
\end{theorem}
\begin{proof}
Let $r_H$ and $R_H$ be the inradius and circumradius of $\triangle H_AH_BH_C$, respectively. It is well known that the circumcircle of the orthic triangle is the nine-point circle of $\triangle ABC$, and hence
\[
R_H = \frac{1}{2}.
\]
We apply the identity
\begin{equation}\label{eq:sumcos}
\cos A + \cos B + \cos C = 1 + \frac{r}{R},
\end{equation}
valid for any triangle, to the orthic triangle $\triangle H_AH_BH_C$. This yields
\[
r_H = \frac12(\cos \alpha_H + \cos \beta_H + \cos \gamma_H - 1).
\]
Now $r_H$ satisfies the orthocenter criterion as the sum
\[
\cos \alpha_H + \cos \beta_H + \cos \gamma_H
\]
satisfies (Theorem \ref{thm:alphah}). This completes the proof.
\end{proof}

\begin{corollary}
Let $\triangle ABC$ be an acute triangle circumscribed about an ellipse whose center \,coincides with the circumcenter of $\triangle ABC$. Then the inradius of the orthic triangle of $\triangle ABC$ is equal to one-half of the harmonic mean of the semi-axes $a$ and $b$ of the ellipse.
\end{corollary}
\begin{proof}
Let $R$ be the circumradius of $\triangle ABC$ and $a$ and $b$ be the semi-axes of the central conic.

By Theorems \ref{thm:alphah}, \ref{thm:congorth}, and \ref{thm:sina}, 
\[
r_H = \frac{ab}{R},
\]
We also have $R=a+b$. See \cite[Corollary 2.2]{Murad2026b}. This gives
\[
r_H = \frac{ab}{a+b},
\]
as claimed.
\end{proof}

\begin{corollary}
Let $\triangle ABC$ be an acute triangle circumscribed about an ellipse one of whose foci coincides with the circumcenter of $\triangle ABC$. Then the inradius of the orthic triangle of $\triangle ABC$ is equal to one-half of the semi-latus rectum of the ellipse.
\end{corollary}
\begin{proof} 
Let $R$ be the circumradius of $\triangle ABC$ and $a$ and $b$ be the semi-axes of the central conic. It follows from Theorem \ref{thm:alphah}, Theorem \ref{thm:congorth}, and Theorem \ref{thm:sina} that 
\[
r_H = \frac{b^2}{R},
\]
We also have $R=2a$. See \cite[Corollary 2.3]{Murad2026b}. This gives
\[
r_H = \frac{b^2}{2a},
\]
as required.
\end{proof}

\begin{theorem}
Let $\triangle ABC \in \mathcal P$ and let $H$ be its orthocenter. Then the sum 
\[
|AH|^2+|BH|^2+|CH|^2
\]
and the product
\[
|AH||BH||CH|
\]
satisfy the orthocenter criterion.
\end{theorem}
\begin{proof}
By Proposition \ref{prop:abcoh},
\[
|AH|^2+|BH|^2+|CH|^2
=
3+|z_H|^2.
\]
On the other hand, Theorem \ref{thm:AH_HHA} and Example \ref{ex:sympol} yield
\[
|AH||BH||CH|=\bigl||z_H|^2-1\bigr|.
\]
Since both quantities are functions of $|z_H|$, the conclusion follows from Theorem \ref{thm:orthocentercriterion}.
\end{proof}

The preceding results naturally lead to another invariant involving the vertices and the midpoints of the sides. Although these quantities are not directly attached to the orthic triangle, they are closely related to the nine-point circle, so we include them here. 

\begin{theorem}\label{thm:midptdist}
Let $M_1,M_2,M_3$ denote the midpoints of the sides of a triangle in $\mathcal P$, and let $O_{\mathcal D}$ be the center of the conic. Then the quantities
\[
|O_{\mathcal D}A|^2+|O_{\mathcal D}B|^2+|O_{\mathcal D}C|^2, \qquad |O_{\mathcal D}M_1|^2+|O_{\mathcal D}M_2|^2+|O_{\mathcal D}M_3|^2
\]
satisfy the orthocenter criterion.

Moreover, if the circumcenter of the triangle coincides with one of the foci of the conic, then
\[
|O_{\mathcal D}M_1|
=
|O_{\mathcal D}M_2|
=
|O_{\mathcal D}M_3|
=
\frac{R}{2},
\]
where $R$ is the circumradius of the triangle.
\end{theorem}

\begin{proof}
Assume that $\mathbb T$ is the common circumcircle of triangles in $\mathcal P$ and $a_1,a_2 \in \mathbb C$ be the foci of the conic. Then
\[
O_{\mathcal D}=\frac{a_1+a_2}{2}.
\]
Hence
\begin{align*}
|O_{\mathcal D}A|^2+|O_{\mathcal D}B|^2+|O_{\mathcal D}C|^2
&=\left|
z_1-\frac{a_1+a_2}{2}
\right|^2+\left|
z_2-\frac{a_1+a_2}{2}
\right|^2+\left|
z_3-\frac{a_1+a_2}{2}
\right|^2 \\
&=
3-\operatorname{Re}((a_1+a_2)z_H)+\frac34
|a_1+a_2|^2.
\end{align*}
Using \eqref{eq:zHrelation} gives
\[
\operatorname{Re}((a_1+a_2)z_H)=\operatorname{Re}((a_1+a_2)^2)+\operatorname{Re}(\overline{a_1a_2}(a_1+a_2)\lambda).
\]
Since $\lambda$ ranges over the admissible parameter set $\Lambda \subset \mathbb T$, the last expression is independent of $\lambda$ if and only if
\[
\overline{a_1a_2}(a_1+a_2)=0,
\]
that is, if and only if
\[
a_1+a_2=0
\qquad\text{or}\qquad
a_1a_2=0.
\]
This proves that the first quantity satisfies the orthocenter criterion.

Similarly, if $M_1$ be the midpoint of the opposite side of the vertex $z_1$. Then
\begin{align}
|O_{\mathcal D}M_1|^2
&:=\left|
\frac{z_2+z_3}{2}-\frac{a_1+a_2}{2}
\right|^2 \nonumber \\
&=
\frac14\left|z_H-z_1-(a_1+a_2)\right|^2  \nonumber\\
&=
\frac14\left|z_1-\overline{a_1a_2} \lambda\right|^2 \label{eq:odm}.
\end{align}
Summing \eqref{eq:odm} cyclically gives
\[
|O_{\mathcal D}M_1|^2+|O_{\mathcal D}M_2|^2+|O_{\mathcal D}M_3|^2
=
\frac34+\frac34|a_1a_2|^2
-\frac12
\operatorname{Re}
\left(
a_1a_2z_H\overline{\lambda}
\right).
\]
Using \eqref{eq:zHrelation} again we obtain
\[
a_1a_2z_H\overline{\lambda}
=
|a_1a_2|^2
+
a_1a_2(a_1+a_2)\overline{\lambda}.
\]
Hence
\[
|O_{\mathcal D}M_1|^2+|O_{\mathcal D}M_2|^2+|O_{\mathcal D}M_3|^2
=
\frac34+\frac14|a_1a_2|^2
-\frac12
\operatorname{Re}
\left(
a_1a_2(a_1+a_2)\overline{\lambda}
\right).
\]
The last expression is independent of $\lambda$ if and only if
\[
a_1+a_2=0
\qquad\text{or}\qquad
a_1a_2=0.
\] 
Thus the second quantity satisfies the orthocenter criterion.

Finally, if $a_2=0$, then \eqref{eq:odm} gives
\[
|O_{\mathcal D}M_1|
=
\frac12,
\]
and similarly
\[
|O_{\mathcal D}M_1|=|O_{\mathcal D}M_1|=\frac12.
\]
Thus every midpoint is at distance $1/2$ from the center of the conic. Restoring the circumradius $R$ gives
\[
|O_{\mathcal D}M_k|=\frac{R}{2},
\qquad k=1,2,3,
\]
which completes the proof.
\end{proof}

\begin{remark}
The second assertion of Theorem \ref{thm:midptdist} also follows from a result established in \cite[Corollary 3.4]{Dragovic-Murad2026}. Indeed, when the circumcenter coincides with a focus of the conic, the nine-point circle of the triangle coincides with the auxiliary circle of the conic. Hence the three side midpoints, which lie on the nine-point circle, are all at distance $R/2$ from the center of the conic.
\end{remark}

\subsection{Polar Circles}\label{sec:polcirc}
The invariance of $|AH|\,|HH_A|$ established in Theorem \ref{thm:AH_HHA} naturally gives rise to a  circle associated with an obtuse triangle. We show that this circle also exhibits similar invariance properties along Poncelet families.

\begin{definition}\label{def:polarcirc}
The \emph{polar circle} of an obtuse triangle $\triangle ABC$ is the circle centered at its orthocenter $H$ with radius $r_p$ satisfying
\[
r_p^2=|AH|\,|HH_A|.
\]
\end{definition}
\begin{remark}
    Since
    \[
    |AH|\,|HH_A|=|BH|\,|HH_B|=|CH|\,|HH_C|,
    \]
    the value of $r_p$ defined above is independent of the choice of the vertex of $\triangle ABC$.
\end{remark}

\begin{theorem}[Area Invariance of Polar Circles in a Poncelet Family]
The area of the polar circle of a triangle in $\mathcal P$ satisfies the orthocenter criterion.

In particular, the polar circle remains fixed throughout the family if and only if the common circumcenter of the triangles coincides with one of the foci of the conic.
\end{theorem}
\begin{proof}
By Definition \ref{def:polarcirc} and Theorem \ref{thm:AH_HHA}, the radius $r_p$ of the polar circle of a triangle in $\mathcal P$ remains invariant if and only if the circumcenter of the triangles coincides either with the center of the conic or with one of its foci.

Since the center of the polar circle is the orthocenter of the reference triangle, by Corollary \ref{cor:focusorthocenter}, the center of the polar circle remains fixed throughout the family if and only if the circumcenter of the reference triangle coincides with one of the foci of the conic. Hence both the center and the radius of the polar circle remain fixed throughout the family, and therefore the area of the polar circle (equivalently, the circle itself) remains invariant.

This completes the proof.
\end{proof}

\begin{corollary}
Let $\triangle ABC$ be an obtuse triangle circumscribed about an ellipse $\mathcal D$ whose center coincides with the circumcenter of $\triangle ABC$. Then the area of the polar circle of $\triangle ABC$ is twice the area of the ellipse $\mathcal D$.
\end{corollary}

\begin{proof}
Since $\triangle ABC$ is obtuse, its polar circle is well defined. By Theorem \ref{thm:AH_HHA},
\[
r_p^2=\frac12\left(|z_H|^2-1\right).
\]
On the other hand, it follows from \cite[Corollary 2.2]{Murad2026b} that
\[
4ab=|z_H|^2-1,
\]
where $a$ and $b$ denote the lengths of the semi-axes of the ellipse. Hence
\[
r_p^2=2ab.
\]
Therefore,
\[
\operatorname{Area}(\text{polar circle})
=\pi r_p^2
=2\pi ab
=2\times\operatorname{Area}(\text{ellipse}),
\]
which proves the result.
\end{proof}

\begin{corollary}
Let $\triangle ABC$ be circumscribed about a hyperbola $\mathcal D$. Assume that the circumcenter of $\triangle ABC$ coincides with one of the foci of $\mathcal D$. Then the area of the polar circle of $\triangle ABC$ is twice the area of the minor auxiliary circle of $\mathcal D$.
\end{corollary}

\begin{proof}
By Corollary \ref{cor:focus}, the triangle $\triangle ABC$ is obtuse, so its polar circle is well defined. By Theorem \ref{thm:AH_HHA},
\[
r_p^2=\frac12\left||z_H|^2-1\right|.
\]
On the other hand,
\[
b^2=\frac14\left||z_H|^2-1\right|,
\]
where $b$ denotes the length of the semi-minor axis of the hyperbola (see \cite[Corollary 2.3]{Murad2026b}). Hence
\[
r_p^2=2b^2.
\]
Since the area of the polar circle is $\pi r_p^2$ and the area of the minor auxiliary circle is $\pi b^2$, the assertion follows.
\end{proof}

\subsection{Tangential Triangles}\label{sec:tangtri}
The identities obtained in the preceding subsections also lead to invariants associated with the tangential triangle. Since the tangential triangle is projectively dual to the original triangle with respect to the circumcircle, its geometry is closely related to the orthocenter parameter $z_H$. We show that several metric properties of the tangential triangle depend only on $z_H$.

Throughout this subsection, for each vertex $X\in \{A,B,C\}$, $T_X$ denotes the intersection of the tangents to the circumcircle of $\triangle ABC$ at its remaining two vertices.
\begin{definition}
The triangle formed by the polars of the vertices of $\triangle ABC$ with respect to a conic is called the \emph{polar triangle} of $\triangle ABC$.

When the conic is the circumcircle of $\triangle ABC$, the polar triangle is called the \emph{tangential triangle}. Its circumcircle is
called the \emph{tangential circle} of $\triangle ABC$.
\end{definition}

\begin{proposition}\label{prop:tangentialcoord}
Let $\triangle z_1z_2z_3$ be inscribed in the unit circle $\mathbb T$. Let $w_1 \in \mathbb C$ denote the complex coordinate of the intersection of the tangents to $\mathbb T$ at $z_2$ and $z_3$. Then
\[
w_1=\frac{2z_2z_3}{z_2+z_3}.
\]
Cyclically,
\[
w_2=\frac{2z_3z_1}{z_3+z_1},\qquad
w_3=\frac{2z_1z_2}{z_1+z_2}.
\]
\end{proposition}
\begin{proof}
The complex number $w_1$ satisfies
\[
w_1\overline{z_2}+\overline{w_1}z_2=2,
\qquad
w_1\overline{z_3}+\overline{w_1}z_3=2.
\]
Subtracting these equations gives
\[
w_1(\overline{z_2}-\overline{z_3})
=
-\overline{w_1}(z_2-z_3).
\]
Since $|z_2|=|z_3|=1$, we have
\[
\overline{z_2}-\overline{z_3}
=
-\frac{z_2-z_3}{z_2z_3},
\]
and hence
\[
\overline{w_1}
=
\frac{w_1}{z_2z_3}.
\]
Substituting this into the first tangent equation yields
\[
w_1\overline{z_2}
+\frac{w_1}{z_3}
=2,
\]
or equivalently,
\[
w_1\left(\frac{1}{z_2}+\frac{1}{z_3}\right)=2.
\]
Therefore,
\[
w_1=\frac{2z_2z_3}{z_2+z_3}.
\]
The formulas for $w_2$ and $w_3$ follow by cyclic permutation of the
indices.
\end{proof}

\begin{theorem}
Let $\triangle ABC\in\mathcal P$, and let
$\triangle T_AT_BT_C$ denote its tangential triangle.
For each vertex $X \in \{A,B,C\}$, let $P_X$ be the point dividing
$\overline{XT_X}$ internally in the ratio $1:2$, and let $O$ denote the circumcenter of $\triangle ABC$. Then the products
\[
|OT_A|\,|OT_B|\,|OT_C|, \qquad
|OP_A|\,|OP_B|\,|OP_C|
\]
satisfy the orthocenter criterion.
\end{theorem}
\begin{proof}
Assume that the circumcircle is the unit circle
$\mathbb T$, so that the circumcenter is the origin $O$. Let $z_1,z_2,z_3 \in \mathbb C$ be the complex coordinates of the vertices $A,B,C$, respectively, 
and let $w_1,w_2,w_3 \in \mathbb C$ denote the vertices of its tangential triangle. By Proposition \ref{prop:tangentialcoord},
\[
w_1=\frac{2z_2z_3}{z_2+z_3},\qquad
w_2=\frac{2z_3z_1}{z_3+z_1},\qquad
w_3=\frac{2z_1z_2}{z_1+z_2}.
\]

Since $|z_i|=1$, we obtain
\[
|w_1|
=
\left|
\frac{2z_2z_3}{z_2+z_3}
\right|
=
\frac{2}{|z_2+z_3|},
\]
and cyclically,
\[
|w_2|
=
\frac{2}{|z_3+z_1|},
\qquad
|w_3|
=
\frac{2}{|z_1+z_2|}.
\]
Hence
\begin{align*}
    |OT_A|\,|OT_B|\,|OT_C|&=|w_1|\,|w_2|\,|w_3|\\
    &=\frac{8}
{|z_1+z_2|\,|z_2+z_3|\,|z_3+z_1|}\\
&=
\frac{8}{\bigl||z_H|^2-1\bigr|}.
\end{align*}
where the last equality followed from Example \eqref{ex:sympol} and $z_H$ denotes the complex coordinate of the orthocenter. 

Now let $p_1,p_2,p_3 \in \mathbb C$ be the complex coordinates of $P_A$, $P_B$, and $P_C$, respectively. Then
\[
p_1=\frac{2z_1+w_1}{3},\qquad
p_2=\frac{2z_2+w_2}{3},\qquad
p_3=\frac{2z_3+w_3}{3}.
\]
A direct computation yields
\[
|OP_A|
:= |p_1|=
\frac{|z_H|}
{3|z_2+z_3|},
\]
and the expressions for $|OP_B|$ and $|OP_C|$ follow cyclically.

Consequently,
\[
|OP_A|\,|OP_B|\,|OP_C|
=
\frac{|z_H|^3}
{27\bigl||z_H|^2-1\bigr|}.
\]
Since the products depend only on the quantity $|z_H|$, they satisfy the orthocenter criterion (Theorem \ref{thm:orthocentercriterion}). 
\end{proof}

\begin{theorem}\label{thm:tangcirc}
Let $\triangle z_1z_2z_3$ be inscribed in $\mathbb T$. Then the tangential circle of $\triangle z_1z_2z_3$ is
\[
\left|z-\frac{2z_H}{|z_H|^2-1}\right|
=
\frac{2}{\bigl||z_H|^2-1\bigr|},
\]
provided $|z_H|\neq1$, where $z_H$ denotes the complex coordinate of the orthocenter of $\triangle z_1z_2z_3$.
\end{theorem}

\begin{proof}
Let $w_1$ denote the complex coordinate of the intersection of the tangents to the circumcircle at $z_2$ and $z_3$. From Proposition \ref{prop:tangentialcoord},
\[
w_1=\frac{2z_2z_3}{z_2+z_3}.
\]
Using  
\[
z_H=z_1+z_2+z_3,
\] 
in the last equation, we obtain
\[
    \frac1{w_1}
=\frac12\left(\overline{z_2}+\overline{z_3}\right)
=\frac12\left(\overline{z_H}-\overline{z_1}\right).
\]
Since $|\overline{z_1}|=1$, we have
\[
\left|\overline{z_H}-\frac{2}{w_1}\right|=1.
\]
Let $z=w_1$. Then
\[
|z_H\bar z-2|=|z|.
\]
Squaring both sides gives
\[
(2-z_H\bar z)(2-\overline{z_H}z)=|z|^2,
\]
which simplifies to
\[
(|z_H|^2-1)|z|^2
-2\overline{z_H}z
-2z_H\bar z
+4=0.
\]
Assuming $|z_H|\neq1$, divide by $|z_H|^2-1$ to obtain
\[
|z|^2
-\frac{2\overline{z_H}}{|z_H|^2-1}z
-\frac{2z_H}{|z_H|^2-1}\bar z
+\frac{4}{|z_H|^2-1}=0.
\]
Comparing this with the standard equation of a circle,
\[
|z-c|^2=r^2,
\]
shows that
\begin{equation}\label{eq:incentan}
    c=\frac{2z_H}{|z_H|^2-1},
\end{equation}
and
\[
r^2
=
|c|^2-\frac{4}{|z_H|^2-1}
=
\frac{4}{(|z_H|^2-1)^2}.
\]
Hence
\[
    r=\frac{2}{\bigl||z_H|^2-1\bigr|},
\]
which proves the theorem.
\end{proof}

\begin{corollary}
Let $\triangle z_1z_2z_3$ be inscribed in $\mathbb T$, and let $\triangle w_1w_2w_3$ be its tangential triangle. If $z_H$ denotes the complex coordinate of the orthocenter of $\triangle z_1z_2z_3$, then the quantity
\[
\left|
\frac{1}{\overline{w_1}}
+\frac{1}{\overline{w_2}}
+\frac{1}{\overline{w_3}}
\right|
\]
satisfies the orthocenter criterion.
\end{corollary}

\begin{proof}
By Proposition \ref{prop:tangentialcoord},
\[
w_1=\frac{2z_2z_3}{z_2+z_3}.
\]
Since $|z_2|=|z_3|=1$, we have
\[
\frac{1}{\overline{w_1}}
=
\frac{\overline{z_2}+\overline{z_3}}{2}.
\]
Summing cyclically yields
\begin{align*}
\frac{1}{\overline{w_1}}
+\frac{1}{\overline{w_2}}
+\frac{1}{\overline{w_3}}
&=
\overline{z_1}
+\overline{z_2}
+\overline{z_3} \\
&=
\overline{z_H}.
\end{align*}
Hence,
\[
\left|
\frac{1}{\overline{w_1}}
+\frac{1}{\overline{w_2}}
+\frac{1}{\overline{w_3}}
\right|
=
|z_H|.
\]
The conclusion now follows immediately from Theorem \ref{thm:orthocentercriterion}.
\end{proof}

\begin{corollary}
The area of the tangential circle of a triangle  in $\mathcal P$ satisfies the orthocenter \,criterion.

In particular, the tangential circle remains fixed throughout $\mathcal P$ if and only if the circumcenter of the triangle coincides with one of the foci of $\mathcal D$. 
\end{corollary}

\begin{proof}
Let $\triangle ABC \in \mathcal P$. By Theorem \ref{thm:tangcirc},
the radius of its tangential circle is
\[
\frac{2}{\bigl||z_H|^2-1\bigr|},
\]
which depends only on $|z_H|$. Note that $|z_H| \neq 1$ (Proposition \ref{prop:obltri}).

It follows from \eqref{eq:incentan} that the center $c$ of the tangential triangle of $\triangle 
ABC$ remains fixed if and only if $z_H$ is. The conclusion therefore follows from
Theorem \ref{thm:orthocentercriterion}.
\end{proof}

The tangential triangle and the orthic triangle of a triangle are homothetic. Consequently, the following theorem is an immediate consequence of Theorem \ref{thm:orthicarearatio}. So we state the result without proof.

\begin{theorem}
The ratio of the area of a triangle in $\mathcal P$ and the area of its tangential triangle satisfies the orthocenter criterion. 
\end{theorem}

\subsection{Power Circle}\label{sec:powcirc}
The power circles constitute another natural family of circles associated with a triangle. In this subsection, we investigate their behavior in a Poncelet family and identify the geometric configurations for which they remain invariant.
\begin{definition}
   The circles each passing through a vertex of a triangle and centered at the midpoint of the opposite side are called the \emph{power circles} of the triangle.
\end{definition}

Corollary \ref{cor:mediansqsum} provides a particularly short alternative proof of Theorem 6.1 from \cite{Dragovic-Murad2026}, which we reformulate and include here for completeness.
\begin{theorem}
 The total area of the power circles of a triangle in $\mathcal P$ satisfies the orthocenter criterion.
\end{theorem}
\begin{proof} We may assume that the common circumcircle of the family $\mathcal P$ is $\mathbb T$, and let $\triangle z_1z_2z_3\in\mathcal P$. Then the total area of the power circles of $\triangle z_1z_2z_3$ is
\[
\operatorname{Area}(\mathcal C_1)
+\operatorname{Area}(\mathcal C_2)
+\operatorname{Area}(\mathcal C_3)
=
\frac{3\pi}{4}(9-|z_H|^2),
\]
where $\mathcal C_k$ denotes the power circle passing through the vertex $z_k$, for $k=1,2,3$.

The conclusion now follows immediately from \eqref{eq:mediansqsum} and Theorem \ref{thm:orthocentercriterion}. In particular,
\[
\operatorname{Area}(\mathcal C_1)
+\operatorname{Area}(\mathcal C_2)
+\operatorname{Area}(\mathcal C_3)
=
\frac{3\pi}{4}(9-|a_1|^4)
\]
whenever the circumcenter coincides with the center of the conic, and
\[
\operatorname{Area}(\mathcal C_1)
+\operatorname{Area}(\mathcal C_2)
+\operatorname{Area}(\mathcal C_3)
=
\frac{3\pi}{4}(9-|a_1|^2)
\]
whenever the circumcenter coincides with one of the foci, say $a_2$.
\end{proof}

\subsection{de Longchamps Circle}\label{sec:delongcirc}

The de Longchamps circle is another classical object associated with a triangle via the power circles.

\begin{definition}\label{def:deLongchampcirc}
Let $\triangle ABC$ be an obtuse triangle. The radical circle of the three power circles of $\triangle ABC$ is called the \emph{de Longchamps circle}. It is centered at the de Longchamps point $L$ and has radius
\[
R_L=4R\sqrt{|\cos A\cos B\cos C|}.
\]
\end{definition}

\begin{theorem}
Let $\triangle ABC \in \mathcal P$.
\begin{itemize}
\item[(a)]
The area of the de Longchamps circle of $\triangle ABC$ satisfies the orthocenter criterion.

\item[(b)]
The de Longchamps circle of $\triangle ABC$ remains fixed throughout the family if and only if the circumcenter of $\triangle ABC$ coincides with one of the foci of $\mathcal D$.
\end{itemize}
\end{theorem}
\begin{proof}
Assume that $\mathcal C=\mathbb T$. It follows from Definition \ref{def:deLongchampcirc} and Theorem \ref{thm:sina} that $R_L$ satisfies the orthocenter criterion.

For the second assertion, the de Longchamps point is the reflection of the orthocenter about the circumcenter. Since the circumcenter is the origin, we have
\[
z_L=-z_H.
\]
Thus the center of the de Longchamps circle is fixed if and only if $z_H$ is fixed. By Theorem \ref{thm:orthocentercriterion}, this occurs precisely when the circumcenter coincides with one of the foci of $\mathcal D$.
\end{proof}

\section{Analytic Construction of Central Inconics}
We now consider an inverse problem. Given an oblique triangle,  construct a central conic inscribed in the triangle whose foci are the circumcenter and orthocenter of the triangle. The \,following proposition shows that such a conic always exists, is unique, and can be written explicitly.

\begin{proposition}\label{prop:construction}
Every oblique triangle admits a unique central conic inscribed in it whose foci are the circumcenter and the orthocenter of the triangle. Moreover, the conic is an ellipse if the triangle is acute and a hyperbola if it is obtuse.
\end{proposition}

\begin{proof} Let $\triangle z_1z_2z_3$ be an oblique triangle inscribed in $\mathbb T$. A central conic with foci $a_1,a_2 \in \mathbb C$ inscribed in $\triangle z_1z_2z_3$ is given by the following equation \cite{Murad2026b}:
\[
\mathcal D:\left||z-a_1|\pm|z-a_2|\right|=|1-\overline{a_1}a_2|
\]
where the positive sign corresponds to an ellipse and the negative sign corresponds to a hyperbola.

If $\triangle ABC$ is acute, then $\mathcal D$ must be an ellipse (Corollary \ref{cor:focus}). Therefore, using $a_1=z_H$ and $a_2=0$, we find 
\[
\mathcal D:|z-z_H|+|z|=1
\]
is the required ellipse.

Similarly, if $\triangle ABC$ is obtuse,
\[
\mathcal D:\vert|z-z_H|-|z|\vert=1
\]
is the required hyperbola. See Figure \ref{fig:construction}.
\end{proof}

\begin{figure}
    \centering
\definecolor{wewdxt}{rgb}{0.43137254901960786,0.42745098039215684,0.45098039215686275}
\begin{subfigure}{0.45\textwidth}
\begin{tikzpicture}[scale=1.5]
\clip(-1.5,-1.2) rectangle (2.5,2.5);
\draw [rotate around={86.4999999999997:(0.6785714285714285,0.8928571428571428)},line width=1.pt] (0.6785714285714285,0.8928571428571428) ellipse (0.9105391988558542cm and 0.6060915267313262cm);
\draw [line width=1.pt,gray] (1.,2.)-- (-1.,1.);
\draw [line width=1.pt,gray] (-1.,1.)-- (2.,-1.);
\draw [line width=1.pt,gray] (2.,-1.)-- (1.,2.);
\draw [rotate around={86.4999999999997:(0.6785714285714285,0.8928571428571428)},line width=1.pt,red] (0.6785714285714285,0.8928571428571428) ellipse (0.9105391988558542cm and 0.6060915267313262cm);
\begin{scriptsize}
\draw [fill=gray] (1.,2.) circle (1.0pt);
\draw[color=black] (1,2.15) node {$A$};
\draw [fill=gray] (-1.,1.) circle (1.0pt);
\draw[color=black] (-1.15,1.099206035725166) node {$B$};
\draw [fill=gray] (2.,-1.) circle (1.0pt);
\draw[color=black] (2.120271496261972,-1.0566106215808306) node {$C$};
\draw [fill=gray] (0.6428571428571429,0.21428571428571427) circle (1.0pt);
\draw[color=black] (0.81,0.3591495712768388) node {$O$};
\draw [fill=gray] (0.7142857142857143,1.5714285714285714) circle (1.0pt);
\draw[color=black] (0.85,1.6247533800435434) node {$H$};
\draw [fill=gray] (0.6785714285714286,0.8928571428571428) circle (1.0pt);
\draw[color=black] (0.82,1.0026769316666886) node {$F$};
\draw[color=black] (-0.12134880909600498,0.9597751076406986) node {$\mathcal{D}$};
\end{scriptsize}
\end{tikzpicture}
    \caption{$\triangle ABC$ is acute. $\mathcal{D}$ is an ellipse.}
    \label{fig:construction(A)}
\end{subfigure}
\begin{subfigure}{0.45\textwidth}
 \begin{tikzpicture}[scale=1]
\clip(-4,-2) rectangle (4,5);
\draw [samples=100,domain=-0.99:0.99,rotate around={-117.25:(0.5357142857142857,1.3928571428571426)},xshift=0.5357142857142857cm,yshift=1.3928571428571426cm,line width=1.pt,red] plot ({1.0412414097936245*(1+(\x)^2)/(1-(\x)^2)},{0.8206518066485169*2*(\x)/(1-(\x)^2)});
\draw [samples=100,domain=-0.99:0.99,rotate around={-117.25:(0.5357142857142857,1.3928571428571426)},xshift=0.5357142857142857cm,yshift=1.3928571428571426cm,line width=1.pt,red] plot ({1.0412414097936245*(-1-(\x)^2)/(1-(\x)^2)},{0.8206518066485169*(-2)*(\x)/(1-(\x)^2)});
\draw [line width=1.pt,gray] (1.,2.)-- (-2.,1.);
\draw [line width=1.pt,gray] (-2.,1.)-- (2.,0.);
\draw [line width=1.pt,gray] (2.,0.)-- (1.,2.);
\begin{scriptsize}
\draw [fill=gray] (1.,2.) circle (1.5pt);
\draw[color=black] (1.2218181818181824,2.0709090909090855) node {$A$};
\draw [fill=gray] (-2.,1.) circle (1.5pt);
\draw[color=black] (-2.2,1.107272727272724) node {$B$};
\draw [fill=gray] (2.,0.) circle (1.5pt);
\draw[color=black] (2.203636363636364,-0.1290909090909092) node {$C$};
\draw [fill=gray] (-0.07142857142857142,0.21428571428571427) circle (1.5pt);
\draw[color=black] (-0.26909090909090816,-0.07454545454545478) node {$O$};
\draw [fill=gray] (1.1428571428571428,2.5714285714285716) circle (1.5pt);
\draw[color=black] (1.2763636363636368,2.8709090909090835) node {$H$};
\draw [fill=gray] (0.5357142857142857,1.3928571428571428) circle (1.5pt);
\draw[color=black] (0.6581818181818189,1.6890909090909043) node {$F$};
\draw[color=black] (0.4945454545454553,-1.2018181818181792) node {$\mathcal{D}$};
\end{scriptsize}
\end{tikzpicture}
    \caption{$\triangle ABC$ is obtuse. $\mathcal{D}$ is a hyperbola.}
    \label{fig:construction(B)}   
\end{subfigure}
    \caption{Proposition \ref{prop:construction}.}
    \label{fig:construction}
\end{figure}

\section{Concluding Remarks}
\noindent The applications developed in this paper exhibit a striking common phenomenon: a wide variety of geometric quantities associated with a Poncelet triangle remain invariant precisely when the circumcenter occupies one of two distinguished positions with respect to the underlying central conic. For the reader's convenience, Table \ref{tab:invariants} summarizes the principal invariance results established in this paper.

\begin{table}[ht]
\centering
\caption{Summary of the principal invariance results established in Section \ref{sec:invsympar}.}
\label{tab:invariants}
\renewcommand{\arraystretch}{1.15}
\begin{tabular}{ccl}
\hline
\textbf{Subsection} && \textbf{Invariant quantity} \\
\hline
\ref{sec:lenghtinv} && $|AB|^2+ |BC|^2+ |CA|^2$\\

\ref{sec:anginv} && $ \sin^2 A+ \sin^2 B+ \sin^2 C$, \qquad $\cos A \cos B \cos C$\\

\ref{sec:anginv} && $\displaystyle \cos(\alpha-\beta)+\cos(\beta-\gamma)+\cos(\gamma-\alpha)$\\

\ref{sec:orthictri} &&
$\displaystyle
\frac{\operatorname{Area}(\triangle ABC)}
{\operatorname{Area}(\triangle H_AH_BH_C)}
$\\

\ref{sec:orthictri} &&
Orthic angle invariant:
$\varepsilon\cos\alpha_H+\cos\beta_H+\cos\gamma_H$\\

\ref{sec:orthictri} &&
$|AH|^2+|BH|^2+|CH|^2$, \qquad $|AH|\,|BH|\,|CH|$\\

\ref{sec:orthictri} &&
$|O_{\mathcal D}A|^2+
|O_{\mathcal D}B|^2+
|O_{\mathcal D}C|^2$, \qquad $|O_{\mathcal D}M_1|^2+
|O_{\mathcal D}M_2|^2+
|O_{\mathcal D}M_3|^2$\\

\ref{sec:polcirc} && Radius (equivalently, area) of the polar circle\\

\ref{sec:tangtri} &&
Radius (equivalently, area) of the tangential circle of a  triangle\\

\ref{sec:tangtri} &&
$|OT_A|\,|OT_B|\,|OT_C|$, \qquad $|OP_A|\,|OP_B|\,|OP_C|$\\

\ref{sec:powcirc} &&
$\displaystyle \operatorname{Area}(\mathcal C_A)+\operatorname{Area}(\mathcal C_B)+\operatorname{Area}(\mathcal C_C)$\\

\ref{sec:delongcirc} &&
Radius (equivalently, area) of the de Longchamps circle\\
\hline
\end{tabular}
\end{table}

All the listed quantities in Table \ref{tab:invariants} satisfy the orthocenter criterion. Consequently, the invariance of any one of them throughout the family $\mathcal P$ is equivalent to the invariance of every other.

\medskip

We conclude this paper by proposing a conjecture that extends \cite[Theorem 3.5]{Dragovic-Murad2026} from triangles to cyclic $n$-gons. 
\begin{conjecture}
Let $\mathcal P$ be the family of $n$-gons
inscribed in a circle and circumscribed about a central conic. Then the sum of the areas of the squares constructed on the sides of a polygon in $\mathcal P$ remains invariant throughout the family if and only if the circumcenter of the polygon coincides either with the center of the conic or with one of its foci.
\end{conjecture}

\vfill

\noindent
\textbf{Funding Declaration.} The author received no financial support for the research, authorship, and/or publication of this article.

\noindent
\thanks{\textbf{Acknowledgments}. The author acknowledges the use of \textit{Mathematica} and \textit{GeoGebra} for \,symbolic and numerical computations, as well as for generating figures in this work.

\bibliographystyle{amsplain}  
\bibliography{references}
\end{document}